\documentclass[11pt]{article}
\usepackage{amsmath,amsthm,amsfonts,amssymb,mathrsfs,epsfig,bm}
\usepackage[usenames]{color}
\usepackage[colorlinks=true]{hyperref}

\parskip        \smallskipamount
\newtheorem{theorem}{Theorem}
\newtheorem{lemma}{Lemma}
\newtheorem{corollary}{Corollary}
\newtheorem{proposition}{Proposition}

\theoremstyle{remark}

\newtheorem{remark}{\bf \em Remark}

\def\N{\mathbb{N}}
\def\Z{\mathbb{Z}}

\def\R{\mathbb{R}}

\def\P{\mathbb{P}}
\def\E{\mathbb{E}}

\def\LL{\mathcal{L}}

\def\sP{\mathsf P}
\renewcommand{\phi}{\varphi}
\renewcommand{\epsilon}{\varepsilon}
\newcommand{\1}{{\text{\Large $\mathfrak 1$}}}
\newcommand{\comp}{\raisebox{0.1ex}{\scriptsize $\circ$}}

\newcommand{\var}{\operatorname{var}}

\definecolor{mygray}{gray}{0.9}
\definecolor{deeppink}{RGB}{255,20,147}
\definecolor{mygreen}{rgb}{0.05, 0.576, 0.03}
\definecolor{myred}{rgb}{0.768, 0.09, 0.09}

\long\def\symbolfootnote[#1]#2{\begingroup
\def\thefootnote{\fnsymbol{footnote}}\footnote[#1]{#2}\endgroup}

\newcommand{\States}{\mathbb{I}}

\def\half{\frac{1}{2}}
\def\*{\allowbreak}
\def\n{^{(n)}}
\newcommand{\pd}[2]{\frac{\partial #1}{\partial #2}}
\def\Lip{\operatorname{Lip}}

\begin{document}
\title{\Large\bf Polynomial approximations to continuous
functions and stochastic compositions}
\date{}
\author{Takis Konstantopoulos
\thanks{\href{mailto:takiskonst@gmail.com}{\tt takiskonst@gmail.com};
Department of Mathematics, Uppsala University, SE-751 06 Uppsala, Sweden;
the work of this author was supported by Swedish Research Council
grant 2013-4688}
\and
Linglong Yuan
\thanks{\href{mailto:yuanlinglongcn@gmail.com}{\tt yuanlinglongcn@gmail.com};
Department of Mathematics, Uppsala University, SE-751 06 Uppsala, Sweden}
\and
Michael A. Zazanis
\thanks{\href{mailto:zazanis@aueb.gr}{\tt zazanis@aueb.gr};
Department of Statistics, 76 Patission St.,
Athens University of Economics, Athens 104 34, Greece}}
\maketitle

\begin{abstract}
This paper presents a stochastic approach to theorems
concerning the behavior of iterations of the Bernstein operator
$B_n$ taking a continuous function $f \in C[0,1]$ to a degree-$n$ 
polynomial when the number of iterations $k$ tends to infinity
and $n$ is kept fixed or when $n$ tends to infinity as well.
In the first instance, the underlying stochastic process
is the so-called Wright-Fisher model, whereas, in the second instance,
the underlying stochastic process is the Wright-Fisher diffusion.
Both processes are probably the most basic ones in mathematical genetics.
By using Markov chain theory and stochastic compositions, 
we explain probabilistically
a theorem due to Kelisky and Rivlin, and by using stochastic
calculus we compute a formula for the application of $B_n$ a number
of times $k=k(n)$ to a polynomial $f$ when $k(n)/n$ tends to 
a constant.
\end{abstract}

\section{Introduction}
About 100 years ago, Bernstein \cite{BER12} introduced a concrete
sequence of polynomials
approximating a continuous function on a compact interval.
That polynomials are dense in the set of continuous functions was shown 
by Weierstrass \cite{WEI1885},
but Bernstein was the first to give a concrete method,
one that has withstood the test of time.
We refer to \cite{PIN00} for a history of approximation theory,
including {\em inter alia}
historical references to Weierstrass' life and work and to the
subsequent work of Bernstein.
Bernstein's approach was probabilistic and is nowadays included in numerous
textbooks on probability theory, see, e.g., \cite[p.\ 54]{SHI84}
or \cite[Theorem 6.2]{BILL95}.

Several years after Bernstein's work, the nowadays known as Wright-Fisher
stochastic model was introduced and proved to be a founding one for the
area of quantitative genetics.
The work was done in the context of Mendelian genetics by
Ronald A.\ Fisher  \cite{FIS22,FIS30}
and Sewall Wright \cite{WRI31}.

This paper aims to explain the relation between the Wright-Fisher model
and the Bernstein operator $B_n$, that takes a function $f \in C[0,1]$
and outputs a degree-$n$ approximating polynomial.
Bernstein's original proof was probabilistic. It is thus natural
to expect that subsequent properties of $B_n$ can also be explained
via probability theory. In doing so, we shed new light
to what happens when we apply the Bernstein operator $B_n$
a large number of times $k$ to a function $f$. In fact, things
become particularly interesting when $k$ and $n$ converge simultaneously
to $\infty$. This convergence can be explained by means of
the original Wright-Fisher model as well as a continuous-time
approximation to it known as Wright-Fisher diffusion.

Our paper was inspired by the Monthly paper of Abel and Ivan \cite{AI09}
that gives a short proof of the Kelisky and Rivlin
theorem \cite{KR67} regarding the limit of the iterates of $B_n$
when $n$ is fixed.
We asked what is the underlying stochastic phenomenon that explains this
convergence and found that it is the composition of independent copies
of the empirical distribution function of $n$ i.i.d.\ uniform
random variables. The composition turns out to be
precisely the Wright-Fisher model. Being a Markov chain with absorbing states,
$0$ and $1$, its distributional limit is a random variable
that takes values in $\{0,1\}$; whence the Kelisky and Rivlin
theorem \cite{KR67}.

Composing stochastic processes is in line with the first author's current
research interests \cite{CK14}.
Indeed, such compositions often turn out
to have interesting, nontrivial, limits \cite{CM15}. Stochastic compositions
become particularly interesting when they explain some natural
mathematical or physical principles. This is what we do, in a particular
case, in this paper.
Besides giving fresh proofs to some phenomena, stochastic compositions
help find what questions to ask as well.

We will specifically provide probabilistic proofs for a number of
results associated to the Bernstein operator \eqref{Bn}.
First, we briefly recall Bernstein's probabilistic proof
(Theorem \ref{Berthm}) that says that $B_n f$ converges uniformly to $f$
as the degree $n$ converges to infinity.
Second, we look at iterates $B_n^k$ of $B_n$, meaning that we compose
$B_n$ $k$ times with itself and give a probabilistic proof
of the Kelisky and Rivlin theorem stating that $B_n^k f$ converges
to $B_1 f$ as the number of iterations $k$ tends to infinity (Theorem
\ref{krthm}).
Third, we exhibit, probabilistically, a geometric rate of convergence
to the Kelisky and Rivlin theorem (Proposition \ref{krRATE}).
Fourth, we examine the limit of $B_n^k f$ when both $n$ and $k$
converge to infinity in a way that $k/n$ converges to a constant 
(Theorem \ref{JOINTLIMSthm}) and
show that probability theory gives us a way to prove and set up
computation methods for the limit for ``simple'' functions
$f$ such as polynomials (Proposition \ref{POLY}).
A crucial step is the so-called Voronovskaya's theorem (Theorem \ref{vorthm})
which gives a rate of convergence to Bernstein's theorem
but also provides the reason why the Wright-Fisher model converges to
the Wright-Fisher diffusion; this is explained in Section  \ref{secwfdiff}.

Regarding notation, we let $C[0,1]$ be the set of continuous functions
$f: [0,1]\to \R$, and $C^2[0,1]$ the set of functions having a continuous
second derivative $f''$, including the boundary points, so $f''(0)$
(respectively, $f''(1)$) is 
interpreted as derivative from the right (respectively, left).
For a bounded function $f:[0,1] \to \R$, we denote by $\|f\|$ the
quantity $\sup_{0\le x \le 1} |f(x)|$.

\section{Recalling Bernstein's theorem}
The Bernstein operator $B_n$ maps any function $f:[0,1]\to \R$
into the polynomial
\begin{equation}
\label{Bn}
B_n f(x) := \sum_{j=0}^n \binom{n}{j}
x^j (1-x)^{n-j} f\left(\frac{j}{n}\right).
\end{equation}
%
We are mostly interested in viewing $B_n$ as an operator on $C[0,1]$.
Bernstein's theorem is:
\begin{theorem}[Bernstein, 1912]
\label{Berthm}
If $f \in C[0,1]$ then $B_n f$ converges uniformly to $f$:
\[
\lim_{n \to \infty} \max_{0\le x \le 1} |B_nf(x)-f(x)| =0.
\]
\end{theorem}
The proof of this theorem is elementary if probability theory is used
and goes like this. Let $X_1, \ldots, X_n$ be independent 
Bernoulli random variables with $\P(X_i=1)=x$, $\P(X_i=0)=1-x$
for some $0\le x \le 1$.
If $S_n$ denotes the number of variables with value $1$ then
$S_n$ has a binomial distribution:
\begin{equation}
\label{binomdist}
\P(S_n=j) = \binom{n}{j} x^j (1-x)^{n-j}, \quad j =0,1,\ldots n.
\end{equation}
Therefore 
\begin{equation}
\label{BnE} 
\E f(S_n/n) = \sum_{j=0}^n f(j/n) \P(S_n=j) = B_n f(x).
\end{equation}
Now let 
\[
m(\epsilon) := \max_{|x-y|\le \epsilon} |f(x)-f(y)|,
\quad 0<\epsilon <1.
\]
Since $f$ is continuous on the compact set $[0,1]$ it is also 
uniformly continuous and so
$m(\epsilon) \to 0$ as $\epsilon \to 0$.
Let $A$ be the event that $|S_n/n- x| \le \epsilon$
and $\1_A$ the indicator of $A$ (a function that is $1$ on $A$ and $0$ on
its complement).
We then write
\[
\E |f(S_n/n) - f(x)|
= \E\left( |f(S_n/n) - f(x)| \1_A\right)
+ \E\left( |f(S_n/n) - f(x)| \1_{A^c}\right)
\le m(\epsilon) + 2 \| f\| \, \P(A^c).
\]
By Chebyshev's inequality,
\[
\P(A^c) = \P(|S_n-nx|\ge n\epsilon|) \le
(n\epsilon)^{-2}\, \E (S_n-n x)^2
= (n\epsilon)^{-2}\, n x(1-x) \le \frac{1}{4} \epsilon^{-2} n^{-1}.
\]
Therefore,
\[
\E |f(S_n/n) - f(x)|
\le m(\epsilon) + \frac{\| f\|}{2\epsilon^2 n}.
\]
Letting $n \to \infty$ the last term goes to $0$ and letting $\epsilon \to 0$
the first term vanishes too, thus establishing the theorem.

\begin{remark}
A variant of Bernstein's theorem due to Marc Kac \cite{KAC37}
gives better estimate if $f$ is Lipschitz or, more generally, H\"older
continuous. Indeed, if $f$ satisfies $|f(x)-f(y)| \le C |x-y|^\alpha$
for some $0<\alpha \le 1$ then,
\[
\E |f(S_n/n) - f(x)|
\le C\, \E |S_n/n - x|^\alpha
\le C\, (\E |S_n/n - x|^2 )^{\alpha/2}
= C (n^{-1} x(1-x))^{\alpha/2}
\le \frac{C\, 2^{-\alpha}}{n^{\alpha/2}},
\]
where the first inequality used the H\"older continuity of $f$,
while the second used Jensen's inequality twice; indeed,
if $Z$ is a positive random variable then
$\E Z^\beta \le (\E Z)^\beta$, by the concavity of the function
$z \mapsto z^\beta$ when $0< \beta < 1$.
\end{remark}
\begin{remark}
H\"older continuous functions with small $\alpha < 1$ are ``rough'' functions.
The previous remark tells us that we may not have a good rate of
convergence for these functions. On the other hand, if $f$ is smooth
can we expect a good rate of convergence? A simple calculation
with $f(x)=x^2$ shows that $B_n f(x) = \E (S_n/n)^2
= (\E(S_n/n))^2 + \var(S_n/n) = x^2 + \frac{1}{n} x(1-x)$.
Excluding the trivial case $f(x)=ax+b$ (the only functions $f$ for
which $B_nf =f$),
can the rate of convergence be better than $1/n$ for some 
smooth function $f$?
No, and this is due to Voronovskaya's theorem (Theorem \ref{vorthm} in
Section \ref{Vsec}).
\end{remark}

\begin{remark}[Some properties of the Bernstein operator]
\label{PJ}
$~$\\(i) $B_n$ is an increasing operator: 
If $f \le g$ then $B_nf \le B_n g$.
(Proof: $B_n f$ is an expectation; see \eqref{BnE}.)
\\
(ii) If $f$ is a convex function then $B_nf \ge f$.
Indeed, $B_n f(x) = \E f(S_n/n) \ge f(\E(S_n/n))$,
by Jensen's inequality, and, clearly, $\E(S_n/n)=x$.
\\
(iii)
If $f$ is a convex function than $B_n f$ is
a also convex.
See Lemma \ref{Bconv} in Section \ref{Isec} for a proof.
\end{remark}

\section{Iterating Bernstein operators}
\label{Isec}
Let $B_n^2 := B_n \comp B_n$ be the composition of $B_n$ with itself
and, similarly, let $B_n^k := B_n \comp \cdots \comp B_n$ ($k$ times).
Abel and Ivan \cite{AI09} give a short proof of the following.
\begin{theorem}[Kelisky and Rivlin, 1967]
\label{krthm}
For fixed $n \in \N$, 
and any function $f: [0,1]\to \R$,
\[
\lim_{k \to \infty} \max_{0 \le x \le 1}|B^k_n f(x) -f(0)-(f(1)-f(0))x| =0.
\]
\end{theorem}
\begin{remark}
Note that this says that $B_n^k f(x) \to B_1 f(x)$, as $k\to \infty$,
uniformly in $x \in [0,1]$.
If $f$ is convex
by Remark \ref{PJ}(ii), $B_n f \ge f$. 
By Remark \ref{PJ}(i), $B_n^k f(x)$ is increasing in $k$
and hence the limit of Theorem \ref{krthm} is actually achieved
by an increasing sequence of convex functions (Remark \ref{PJ}(iii)).
\end{remark}
\begin{proof}[A probabilistic proof for Theorem \ref{krthm}] 
To prepare the ground, we construct
the earlier Bernoulli random variables in a different way. 
We take $U_1, \ldots, U_n$ to be independent
random variables, all uniformly distributed on the interval $[0,1]$,
and their {\em empirical distribution function}
\begin{equation}
\label{Gdef}
G_n(x) := \frac{1}{n} \sum_{j=1}^n \1_{U_j \le x} = \frac{1}{n}S_n(x),
\quad 0 \le x \le 1.
\end{equation}
We shall think of $G_n$ as a random function.
Note that, for each $x$, $S_n(x)$ has the binomial 
distribution \eqref{binomdist}.
The advantage of the current representation is that $x$, instead of being
a parameter of the probability distribution, is now an explicit parameter
of the new random object $G_n(x)$.
We are allowed to (and we will) pick a sequence of independent copies
$G_n^1, G_n^2, \ldots$ of $G_n$.
For a positive integer $k$ let
\begin{equation}
\label{Amarkov}
H^k_n := G_n^k \comp G_n^{k-1} \comp \cdots \comp G_n^1
\end{equation}
be the composition of the first $k$ random functions. So $H_n^k$ is
itself a random function.
By using the independence and the definition of $B_n$ we have that
(See also Section A1 in the Appendix)
\begin{equation}
\label{Bkf}
\E f(H_n^k(x)) = B_n^k f(x), \quad 0 \le x \le 1,
\end{equation}
for any function $f$.
Hence
the limit over $k \to \infty$ of the right-hand side
is the expectation of the limit of the random variable $f(H_n^k(x))$,
if this limit exists.
(We make use of the fact that \eqref{Bkf} is a finite sum!)
To see that this is the case, we fix $n$ and $x$ and consider the
sequence
\[
H_n^k(x), \quad k=1,2,\ldots
\]
with values in 
\[
\States_n:= \left\{0, \frac{1}{n}, \ldots, \frac{n-1}{n}, 1 \right\}.
\]
We observe that this has the Markov property,\footnote{Admittedly, 
it is a bit unconventional to use an upper index
for the time parameter of a Markov chain but, in our case,
we keep it this way because it appears naturally in the
composition operation.}
namely, $H_n^{k+1}(x)$ is independent of $H_n^1(x),\*\ldots,\* H_n^{k-1}(x)$,
conditional on $H_n^k(x)$.
By \eqref{binomdist},
the one-step transition probability of this Markov chain is
\begin{equation}
\label{tranprob}
p\left(\frac{i}{n},\frac{j}{n}\right):=
\P\left(H_n^{k+1}(x) = \frac{j}{n} \mid H_n^k(x) = \frac{i}{n}\right)
=\binom{n}{j}\left(\frac{i}{n}\right)^j \left(1-\frac{i}{n}\right)^{n-j}.
\end{equation}
Since $p(0,0)=p(1,1)=1$, states $0$ and $1$ are absorbing, whereas
for any $x \in \States_n\setminus\{0,1\}$, we have $p(x,y) > 0$
for all $y \in \States_n$.
Define the absorption time
\[
T(x):=\inf\big\{k \in \N:\, H^k_n(x) \in \{0,1\}\big\}.
\]
Elementary Markov chain theory
\cite[Ch.\ 11]{GS97} 
tells us that
\begin{lemma}
\label{MarkovElem}
For all $x$, $\P(T(x)<\infty)=1$.
\end{lemma}
Therefore, with probability $1$, we have that
\[
H_n^k(x) = H^{T(x)}_n(x)=:W(x), \quad \text{for all but finitely many $k$}.
\]
Hence $f(H_n^k(x)) = f(W(x))$ for all bur finitely many $k$ and so
\[
\lim_{k \to \infty} \E f(H_n^k(x)) = \E f(W(x)).
\]
But the random variable $W(x)$ takes two values: $0$ and $1$.
Notice that $\E H_n^k(x)=x$ for all $k$. Hence $\E W(x)=x$.
But $\E W(x) = 1\times \P(W(x)=1) + 0 \times \P(W(x)=0)$. Thus
$\P(W(x)=1)=x$, and $\P(W(x)=0)=1-x$.
Hence
\begin{equation} \label{B1}
\lim_{k \to \infty} \E f(H_n^k(x)) = f(0) (1-x) + f(1) x.
\end{equation}
This proves the announced limit of Theorem \ref{krthm}
but without the uniform convergence.
However, since all polynomials of the sequence are of
degree at most $n$, and $n$ is a fixed number, convergence for each $x$ implies
convergence of the coefficients of the polynomials.
\end{proof}

To prove a rate of convergence in Theorem \ref{krthm} we need to
show what was announced in Remark \ref{PJ}(iii). Recall that $G_n(x)=
S_n(x)/n$.
\begin{lemma}[Convexity preservation]
\label{Bconv}
If $f$ is convex then so is $B_n f$.
\end{lemma}
\begin{proof}
Let $f:[0,1]\to \R$ and $\epsilon > 0$.
We shall prove that
\begin{equation}
\label{Bfder}
\frac{d}{dx} B_n f(x) 
= n \E\left[f\left(\frac{S_{n-1}(x)+1}{n}\right) - f\left(\frac{S_{n-1}(x)}{n}\right)\right]
\end{equation}
This can be done by direct computation using \eqref{binomdist}.
Alternatively, we can give a probabilistic argument.
Consider $B_n f(x+\epsilon) - B_n f(x) = \E [f(G_n(x+\epsilon)) - f(G_n(x))]$
and compute first-order terms in $\epsilon$.
By \eqref{Gdef},
$f(G_n(x+\epsilon)) - f(G_n(x))$ is nonzero if and only
if at least one of the $U_i$'s falls in the
interval $(x,x+\epsilon]$. The probability that $2$ or more 
variables fall in this interval is $o(\epsilon)$, as $\epsilon \to 0$.
Hence, if $F_\epsilon$ is the event that {\em exactly} one of the
variables falls in this interval, then
\begin{equation}
\label{init}
B_n f(x+\epsilon) - B_n f(x)
= \sum_{k=0}^{n-1} [f((k+1)/n)-f(k/n)]\,\P(S_n(x)=k,F_\epsilon)+o(\epsilon).
\end{equation}
If we let $F_{j,\epsilon}$ be the event that only $U_j$ is
in $(x,x+\epsilon)$, then 
$\P(S_n(x)=k, F_{j,\epsilon})$ is independent of $j$, so
$\P(S_n(x)=k, F_\epsilon) = n \P(S_n(x)=k, F_{n,\epsilon})
=n \epsilon\, \P(S_{n-1}(x)= S_{n-1}(x+\epsilon)=k)= 
n \epsilon\, \P(S_{n-1}(x)=k) \, (1-\epsilon/(1-x))^{n-1-k}
= (n \epsilon +o(\epsilon) )\, \P(S_{n-1}(x)=k)$.
So \eqref{init} becomes
\[
B_n f(x+\epsilon) - B_n f(x)
= n\epsilon \E\left[f\left(\frac{S_{n-1}(x)+1}{n}\right) - f\left(\frac{S_{n-1}(x)}{n}\right)\right]
+ o(\epsilon),
\]
and, upon 
dividing by $\epsilon$ and letting $\epsilon\to 0$, we obtain \eqref{Bfder}. 
Applying the same formula \eqref{Bfder} once more (no further work is needed)
we obtain
\[
\frac{d^2}{dx^2} B_n f(x) 
= n(n-1) \E\left[f\left(\frac{S_{n-2}(x)+2}{n}\right) 
- 2 f\left(\frac{S_{n-2}(x)+1}{n}\right) 
+f\left(\frac{S_{n-2}(x)}{n}\right)\right].
\]
Bring in now the assumption that 
$f$ is convex, whence $f(y+(2/n)) - 2 f(y+(1/n))+f(y) \ge 0$,
 for all $0 \le y \le 1-(2/n)$, and deduce that $(B_nf )''(x) \ge 0$
for all $0\le x \le 1$.
So $B_n f$ is a convex function.
\end{proof}

\begin{center}
\begin{tabular}{cc}
\epsfig{file=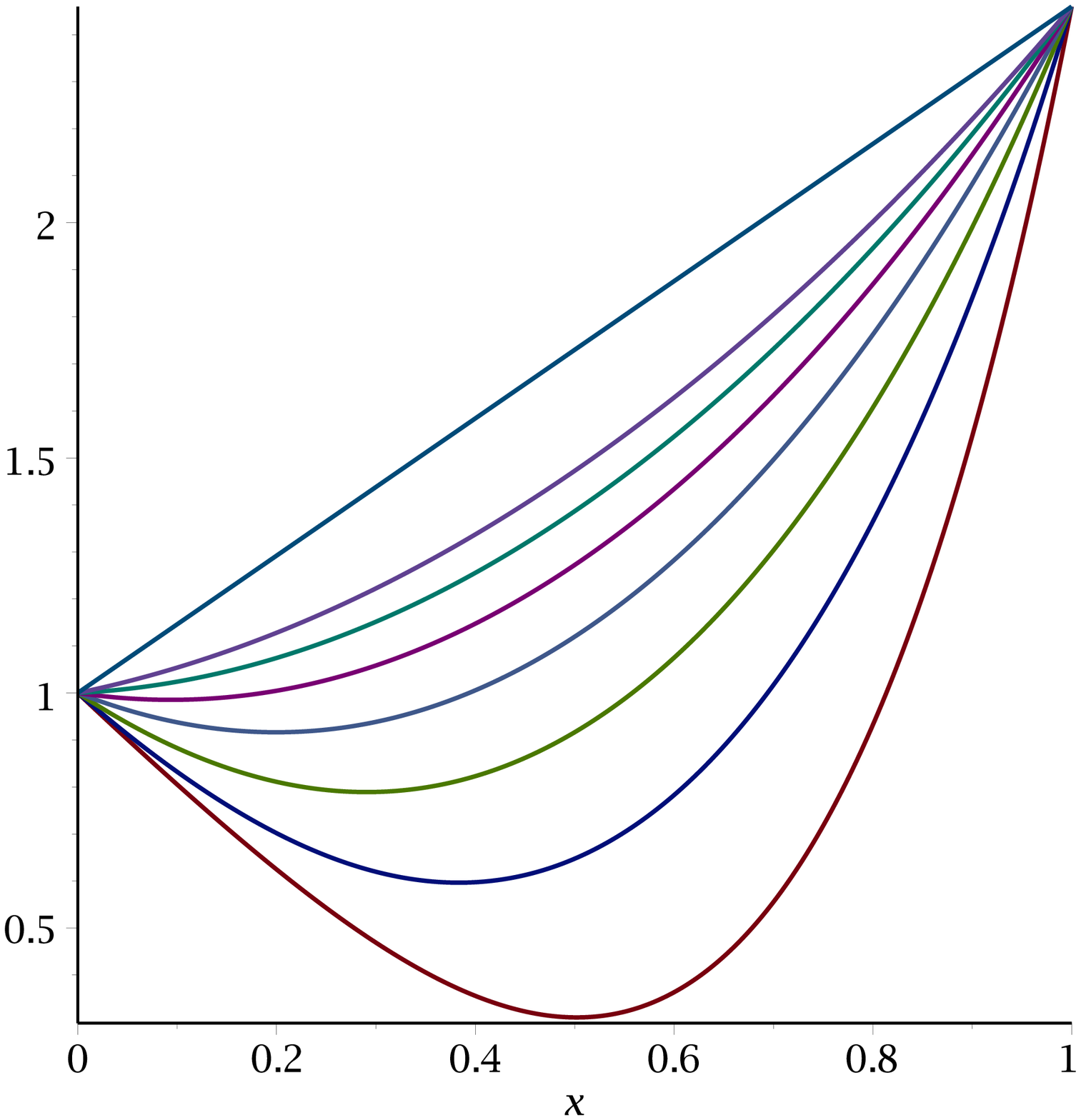,width=5cm} 
&
\epsfig{file=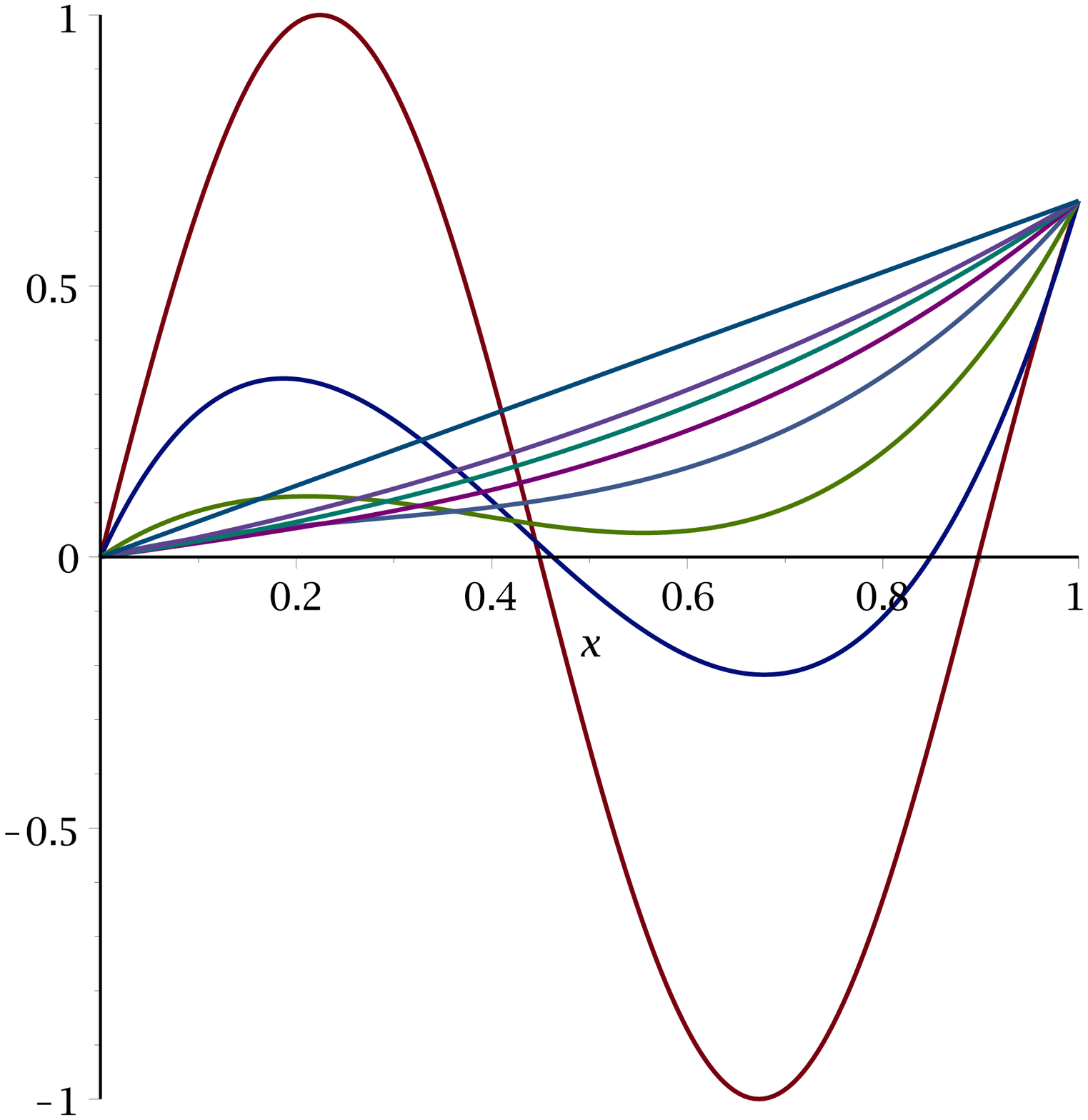,width=5cm}
\end{tabular}
\\
{\small \em Convergence of the iterates of $B^k_n f$ as $k \to \infty$
\\
for a convex $f$ (left) and a nonconvex one (right).}
\end{center}

We can now exhibit a rate of convergence.
\begin{proposition}
\label{krRATE}
For all $0\le x \le 1$, $k, n \in \N$,
\begin{equation}
\label{betabound}
|B_n^k f(x) - B_1 f(x)| \le 2 \|f\|\, \beta(k,x),
\end{equation}
where
\begin{equation}
\label{naprob}
\beta(k,x):=\P(H_n^k(x) \not \in \{0,1\}).
\end{equation}
Moreover,
\begin{equation}
\label{betahalf}
\beta(k,x) \le \beta(k,1/2)
\end{equation}
and
\begin{equation}
\label{betaup}
\beta(k,x) \le n \left(1-\frac{1}{n}\right)^{k-1}\, x\,(1-x).
\end{equation}
\end{proposition}
\begin{proof}
We have, for all positive integers $\ell$ and $k$,
\begin{align*}
 B_n^k f(x) - B_n^{k+\ell} f(x) 
&= \E\left[ f(H_n^{k+\ell}(x)) - f(H_n^k(x))\right]
\\
&= \E\left[ f(G_n^{k+\ell} \comp \cdots\comp  G_n^{k+1}(H_n^k(x))) - f(H_n^k(x))\right]
\\
&= \sum_{y\in \States_n\setminus\{0,1\}} \P(H_n^k(x)=y)\,
\E\left[ f(G_n^{k+\ell}\comp  \cdots\comp  G_n^{k+1}(y)) - f(y)\right],
\end{align*}
whence
\begin{align*}
\left| B_n^k f(x) - B_n^{k+\ell} f(x) \right| &\le  2\|f\|\, \beta(k,x).
\end{align*}
Letting $\ell \to \infty$ and using Theorem \ref{krthm} 
yields \eqref{betabound}.
To prove \eqref{betahalf}, we notice that
$\beta(k,x) = \E \phi(H_n^k(x))$, where $\phi(x)=0$
if $x\in\{0,1\}$ and $1$ otherwise,
and, by \eqref{Bkf}, $\beta(k,x) = B^k_n \phi(x)$.
Note that $\beta(1,x) = B_n \phi(x)=1-x^n-(1-x)^n$ is a concave function.
By \eqref{Bkf}, $\beta(k,x) = B^k_n \phi(x) = B_n^{k-1} B_n \phi(x)$
which is also concave by Lemma \ref{Bconv}.
Hence
\[
\half \beta(k,x) + \half \beta(k,1-x) \le \beta(k,1/2).
\]
Since, by symmetry, $\beta(k,x)=\beta(k,1-x)$, inequality \eqref{betahalf} 
follows.
For the final inequality \eqref{betaup}, notice that
\[
\beta(k,x) = 1 - \E\left[(H_n^{k-1}(x))^n + (1-H_n^{k-1}(x))^n\right]
\le n\, \E \left[H_n^{k-1}(x) \left(1-H_n^{k-1}(x)\right)\right],
\]
where we used the inequality $1-t^n -(1-t)^n \le n t(1-t)$, for all $0 \le t
\le 1$.
Therefore,
\[
\beta(k,x) \le n\, \E \left[H_n^{k-1}(x) \left(1-H_n^{k-1}(x)\right)\right]
=: n\, \gamma(k-1,x).
\]
Using
\[
\E \left[G_n(x) \left(1-G_n(x)\right)\right]
=\left(1-\frac{1}{n}\right)\, x\, (1-x),
\]
we obtain the recursion
\[
\gamma(k-1,x) = \left(1-\frac{1}{n}\right) \gamma(k-2,x).
\]
Taking into account that $\gamma(0,x) = x(1-x)$ we find that
$\gamma(k-1,x) = \left(1-\frac{1}{n}\right)^{k-1}\, x\,(1-x)$.
\end{proof}
%
\begin{remark}
Combining \eqref{betabound} and \eqref{betaup}
we get 
\[
|B_n^k f(x) - B_1 f(x)| \le 2 \|f\|\, n \left(1-\frac{1}{n}\right)^{k-1} x(1-x).
\]
This should be compared with \cite[Eq.\ (4)]{AI09} that says that
$|B_n^k f(x) - B_1 f(x)| \le M(f,n)\, \left(1-\frac{1}{n}\right)^{k-1} x(1-x)$,
for some constant $M(f,n)$ 
which has not been computed in \cite{AI09}, 
whereas we have an explicit constant $2\|f\| n$.
Now, the factor $n$ is probably wasteful and this comes from the fact
that the inequality 
$1-t^n -(1-t)^n \le n t(1-t)$ is not good when $n$ is large. We
only used it because of the simplicity of the right-hand side
that enabled us to compute $\gamma(k,x)$ very easily.
We have a better inequality, namely \eqref{betabound}, but to make 
it explicit one needs to compute $\P(H_n^k(x)=0)$.
\end{remark}

\section{Interlude: population genetics and the Wright-Fisher model}
\label{interlude}
We now take a closer look at the Markov chain described by
the sequence $(H_n^k(x), k \in \N)$ for fixed $n$.
We repeat formula \eqref{tranprob}:
\[
\P\left(n H_n^{k+1}(x) = j \mid n H_n^k(x) = i\right)
=\binom{n}{j}\left(\frac{i}{n}\right)^j \left(1-\frac{i}{n}\right)^{n-j}.
\tag{7$'$}
\label{tranprobbis}
\]
We recognize that it describes 
the simplest stochastic model for reproduction in population
genetics that goes as follows.
There is a population of $N$ individuals each of which 
carries 2 genes. Genes come in 2 variants, I and II, say,
Thus, an individual may have 2 genes of type I both, or of type II both,
or one of each. 
Hence there are $n=2N$ genes in total.
We observe the population at successive generations
and assume that generations are non-overlapping.
Suppose that the $k$-th generation consists of $i$ genes of type I
and $n-i$ of type II.
In Figure \ref{figgener} below, type I genes are yellow,
and type II are red.
\begin{figure}[h]
\begin{center}
\epsfig{file=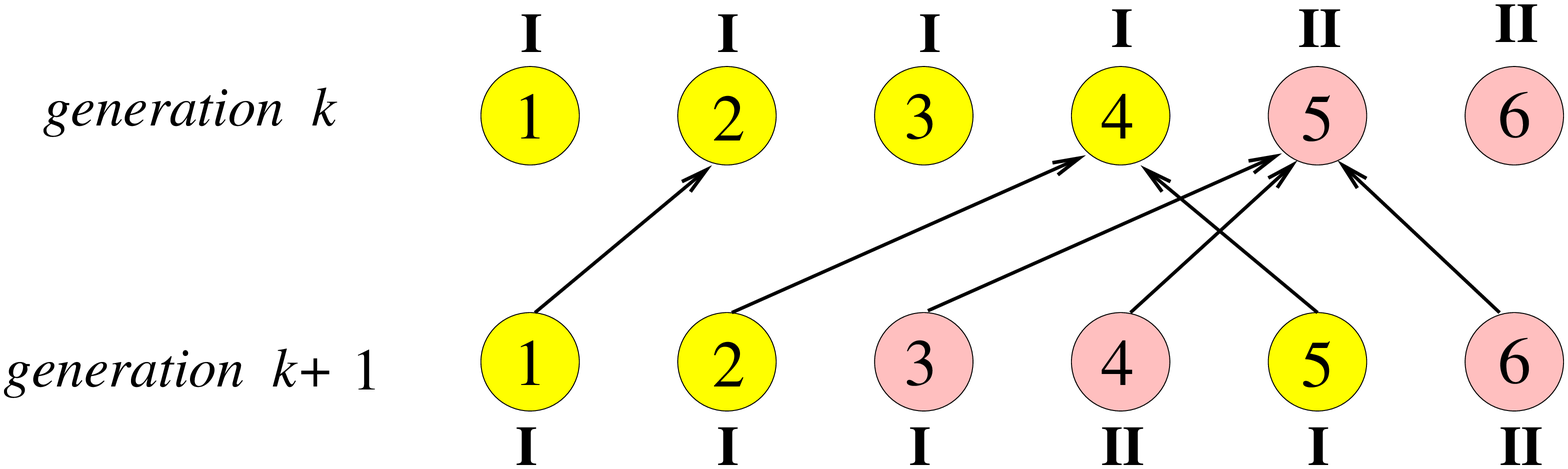,width=9.5cm}
\begin{minipage}{12cm}
\caption{\small\em An equivalent way to form generation $k+1$ from
generation $k$ is by letting $\phi$ be a random element of the set
$[n]^{[n]}$, the set of all mappings from $[n]:=\{1,\ldots,n\}$ into
itself and by drawing an arrow with starting vertex $i$ end ending vertex
$\phi(i)$. If $\phi(i)$ is of type {\rm I} (or {\rm II}) then $i$
becomes {\rm I} (or {\rm II}) too.
}
\label{figgener}
\end{minipage}
\end{center}
\end{figure}
To specify generation $k+1$, we let each gene of generation $k+1$
select a single ``parent'' at random from the genes of the previous generation.
The gene adopts the type of its parent.
The parent selection is independent across genes.
The probability that a specific gene selects a parent of type I
is $i/n$. Since we have $n$ independent trials,
the probability that generation $k+1$ will contain
$j$ genes of type I is given by the right-hand side of 
formula \eqref{tranprobbis}.
If we start the process at generation $0$
with genes of type I being chosen, independently, with
probability $x$ each, then the number of alleles at the $k$-th generation
has the distribution of $n H_n^k(x)$.

This stochastic model we just described is known as the Wright-Fisher model,
and is fundamental in mathematical biology for populations of fixed size.
The model is very far from reality, but has nevertheless been extensively
studied and used.

Early on, Wright and Fisher observed that performing exact computations
with this model is hard. They devised a continuous
approximation observing that the probability
$\P(H_n^k(x) \le y)$ as a function of $x$, $y$ and $k$ can,
when $n$ is large, be approximated by a smooth function
of $x$ and $y$.
(See, e.g., Kimura \cite{KIM64} and the recent paper
by Tran, Hofrichter, and Jost \cite{THJ13}.)
Rather than approximating this probability, we follow
modern methods of stochastic analysis in order to
approximate the discrete stochastic process $(H_n^k(x), k \in \N)$
by a continuous-time continuous-space stochastic process
that is nowadays known as Wright-Fisher diffusion.

\section{The Wright-Fisher diffusion}
\label{secwfdiff}
Our eventual goal is to understand what happens when we consider the
limit of $B^k_n f$, when both $k$ and $n$ tend to infinity.
From the Bernstein and the Kelisky-Rivlin theorems we should
expect that the order at which limits over $k$ and $n$ are taken
matters.
It turns out that the only way to obtain a limit is when the 
ratio $k/n$ tends to a constant, say, $t$.
This is intimately connected to the Wright-Fisher diffusion
that we introduce next.
We assume that the reader has some knowledge of stochastic calculus,
including the It\^o formula and stochastic differential equations
driven by a Brownian motion at the basic level of {\O}ksendal
\cite{OKS03} or at the more advanced level of Bass \cite{BASS}.

We first explain why we expect that the Markov chain $H^k_n(x)$, $k\in
\Z_+$ has a limit, in a certain sense, as $n \to \infty$.
Our explanation here will be informal. We shall give rigorous proofs
of only what we need in the following sections.

The first thing we do is to compute the expected variance of the
increment of the chain, 
and examine whether it converges to zero and at which rate:
see \eqref{STEPVAR}, Theorem \ref{MC2DIFF}, 
Section A3 in the Appendix.
The rate of convergence gives us the right time scale.
Our case at hand is particularly simple because we have
an exact formula:
\begin{equation}
\label{STEPvar}
\E\left[(H^{k+1}_n(x)-H^k_n(x))^2 \mid H^k_n(x)=y\right]
= \E\left[\left(G_n(y)-y\right)^2\right] = \frac{1}{n} y (1-y).
\end{equation}
This suggests that the right time scale at which we should run the Markov chain
is such that the time steps are of size $1/n$.
In other words, consider the points 
\begin{equation}
\label{POINTS}
\text{$(0,x)$, $(1/n, H^1_n(x))$, $(2/n, H^2_n(x))$, $\ldots$}
\end{equation}
and 
draw the random curve
\[
t \mapsto H_n^{[nt]}(x)
\]
(where $[nt]$ is the integer part of $nt$) as in Figure \ref{figcurve}. 
This is at the right time scale.
\begin{figure}[h]
\begin{center}
\epsfig{file=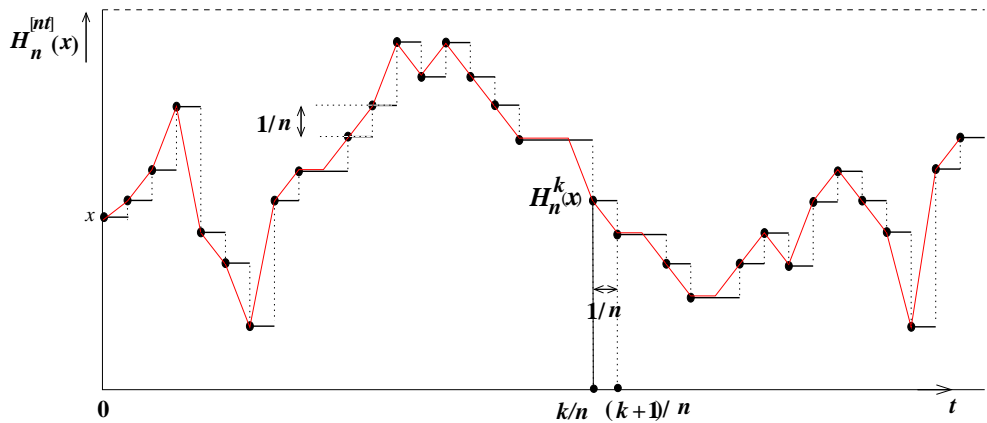,width=13cm}
\begin{minipage}{15cm}
\caption{\small \em A continuous-time curve from the discrete-time Markov
chain}
\label{figcurve}
\end{minipage}
\end{center}
\end{figure}

The second thing we do
(see \eqref{STEPMEAN}, Theorem \ref{MC2DIFF}, Section A3)
is to compute the expected change of the Markov
chain. In our case, this is elementary:
\begin{equation}
\label{STEPmean}
\E\left[H^{k+1}_n(x)-H^k_n(x) \mid H^k_n(x)=y\right]
= \E\left[G_n(y)-y\right] = 0.
\end{equation}
The functions $\sigma(y)^2 = y(1-y)$ and $b(y) = 0$ obtained in
\eqref{STEPvar} and \eqref{STEPmean} suggest that the limit of 
the random curve $(H_n^{[nt]}(x), \, t \ge 0)$
should be a diffusion process $X_t(x)$, $t \ge 0$, satisfying the stochastic
differential equation
\begin{align}
d X_t(x) &= \sigma(X_t(x)) dW_t + b(X_t(x)) dt
\nonumber
\\
&= \sqrt{X_t(x)\, (1-X_t(x))}\, dW_t,
\label{SDEX}
\end{align}
with initial condition $X_0(x)=x$, where $W_t$, $t\ge 0$, 
is a standard Brownian motion.

It is actually possible to prove that $(H_n^{[nt]}, t \ge 0)$ 
converges {\em  weakly} to $(X_t(x), t \ge 0)$,
but this requires an additional estimate on the
size of the increments of the Markov chain that
is, in our case, provided by the following inequality:
for any $\epsilon > 0$,
\begin{equation}
\label{STEPexceed}
\P(|H_n^{k+1}(x) - H_n^k(x)| > \epsilon \mid H_n^k(x)=y)
= \P(|G_n(y)-y| > \epsilon) \le 2 e^{-\half \epsilon^2 n}.
\end{equation}
To see this, apply Hoeffding's inequality (see \eqref{hoeffding} 
Section A2 in the Appendix).

We thus have that \eqref{STEPvar}, \eqref{STEPmean}, \eqref{STEPexceed}
are the conditions \eqref{STEPVAR}, \eqref{STEPMEAN} and \eqref{STEPEXCEED}
of Theorem \ref{MC2DIFF}, Section A3.
In addition, it can be shown that the stochastic differential equation
\eqref{SDEX} admits a unique strong solution for any initial condition $x$.
This is, e.g., a consequence of the Yamada-Watanabe theorem \cite[Theorem 24.4]{BASS}.
Hence, by Theorem \ref{MC2DIFF}, the sequence of continuous random curves
$(H_n^{[nt]}(x), t \ge 0)$ converges weakly to the continuous random function
$(X_t, t \ge 0)$.


One particular conclusion of weak convergence is 
that $\E f(H_n^{[nt]}(x)) \to \E f(X_t(x))$
for any $f \in C[0,1]$ or, equivalently, that
\begin{theorem}[joint limits theorem]
\label{JOINTLIMSthm}
For any $f \in C[0,1]$ and any $t \ge 0$,
\begin{equation}
\label{fidiconv}
\lim_{n \to \infty} B_n^{[nt]} f(x) = \E f(X_t(x)),\quad \text{uniformly in } x.
\end{equation}
\end{theorem}
Since understanding the theorem of Stroock and Varadhan requires
advanced machinery, we shall prove
Theorem \ref{JOINTLIMSthm} directly.
The proof is deferred until Section \ref{JOINTsec}.
The nice thing with this theorem is that we have a way to compute the
limit by means of stochastic calculus, the tools of which
we shall assume as known.

Let $f$ be a twice-continuously differentiable function.
Then, It\^o's formula (\cite[Ch.\ 11]{BASS}, \cite[Ch.\ 4]{OKS03}) says that
\begin{equation}
\label{ITO}
f(X_t) = f(X_0) + \int_0^t f'(X_s) dX_s
+ \frac{1}{2} \int_0^t f''(X_s) X_s(1-X_s) ds.
\end{equation}
If we let
\begin{equation}
\label{generator}
\LL f := \frac{1}{2} x (1-x) \frac{d^2 f}{dx^2},
\end{equation}
take expectations in \eqref{ITO}, and differentiate with respect to $t$,
we obtain 
\begin{equation}
\label{BACKeq}
\pd{}{t} \E f(X_t(x)) = \E (\LL f)(X_t(x)),
\end{equation}
the so-called forward equation of the diffusion.
Now let, for all $s \ge 0$, 
\[
\sP_s g (x) := \E g(X_s(x)),
\]
noticing that $\sP_s g$ is defined for all bounded and
measurable $g$ and that $\P_0 g(x) = g(x)$.
If $g$ is such that $\sP_s g \in C^2$ then we can set $f=\sP_s g$
in \eqref{BACKeq}. Now, $\E(\sP_s g)(X_t(x)) = \E g(X_{t+s}(x))$,
because of the Markov property of $X_t$, $t \ge 0$, and  so
\eqref{BACKeq} becomes 
\[
\pd{}{t} \E g(X_{t+s}(x)) = \E (\LL \sP_s g)(X_t(x)).
\]
Letting $t \to 0$, we 
arrive at the backward equation
\begin{equation}
\label{FRONTeq}
\pd{}{s} \E g(X_s(x)) = \LL \sP_s g(x),
\end{equation}
which is valid if $\sP_s g$ is twice continuously differentiable.
The class of functions $g$ such that both $g$ and $\sP_s g$ are 
in $C^2$ is nontrivial in our case. It contains, at least
polynomials. This is what we show next.

\section{Moments of the Wright-Fisher diffusion}
It turns out that in order to prove Theorem \ref{JOINTLIMSthm} 
we need to compute $\E f(X_t(x))$ when $f$ is a polynomial.
\begin{proposition}
\label{POLY}
For a positive integer $r$, the following holds for the
Wright-Fisher diffusion:
\[
\E X_t(x)^r = \sum_{i=1}^r b_{i,r}(t)\, x^i,
\]
where
\begin{gather}
b_{i,r}(t) = \sum_{j=i}^r \frac{A_{i,r}}{B_{i,j,r}}\, e^{-\alpha_j t},
\label{BIR}
\\
\alpha_j = \half j (j-1), \quad
A_{i,r} = \prod_{k=i+1}^r \alpha_k, \quad
B_{i,j,r} = \prod_{k=i,k \neq j}^r (\alpha_k - \alpha_j),
\quad 1 \le i \le j \le r
\label{coeff}
\end{gather}
(where, as usual, a product over an empty set equals $1$).
\end{proposition}
\begin{proof}
Write $X_t=X_t(x)$ to save space.
By It\^o's formula \eqref{ITO} applied to $f(x)=x^r$,
\[
X_t^r = x^r + r \int_0^t X_s^{r-1} dX_s
+ \half r(r-1) \int_0^t X_s^{r-2} X_s(1-X_s) ds.
\]
Since the first integral is (as a function of $t$) a martingale
starting from $0$ its expectation is 0.
Thus, if we let
\[
m_r(t,x) := \E X_t(x)^r,   
\]
we have
\[
m_r(t,x) = x^r + \alpha_r \int_0^t (m_{r-1}(s,x)-m_r(s,x)) ds.
\]
Thus, $m_1(t,x) = x$, as expected, and
\[
\frac{\partial}{\partial t} m_r(t,x) = \alpha_r
(m_{r-1}(t,x)-m_r(t,x)), \quad r=2,3,\ldots
\]
Defining the Laplace transform
\[
\widehat m_r(s,x) = \int_0^\infty e^{-st} m_r(t,x) dt,
\]
and using integration by parts to see that
$\int_0^t e^{-st} \frac{\partial}{\partial t} m_r(t,x) dt
= s \widehat m_r(s,x) - m_r(0,x) = s\, \widehat m_r(s,x) - x^r$
we have
\[
s\,\widehat m_r(s,x)- x^r = \alpha_r (\widehat m_{r-1}(s,x)-\widehat m_r(s,x)).
\]
Iterating this easy recursion yields
\[
\widehat m_r(s,x)
= \sum_{i=1}^r
\frac{\alpha_r}{s+\alpha_r} \cdots \frac{\alpha_{i+1}}{s+\alpha_{i+1}}
\frac{x^i}{s+\alpha_i}
= \sum_{i=1}^r A_{i,r} \sum_{j=i}^r \frac{1/B_{i,j,r}}{s+\alpha_j}\, x^i,
\]
where the second equality was obtained by partial fraction expansion
(and the notation is as in \eqref{coeff}).
Since the inverse Laplace transform of $1/(s+\alpha_j)$ is $e^{-\alpha_j t}$,
the claim follows.
\end{proof}
\begin{remark}
Formula \eqref{BIR} was proved by Kelisky and Rivlin \cite[Eq.\ (3.13)]{KR67}
and Karlin and Ziegler \cite[Eq.\ (1.13)]{KZ70}
by entirely different methods.
(the latter paper contains a typo in the formula).
Eq.\ (3.13) of \cite{KR67} reads:
\begin{equation}
\label{BIR2}
b_{i,r}(t) = \frac{i}{r}  {r\choose i}^{2}\sum _{j=i}^{r}
\frac {\displaystyle  \left( -1 \right) ^{i+j} {r-i\choose j-i}^{2}{
}}{\displaystyle {2\,j-2\choose j-i}{j+r-1
\choose r-j}}\,
e^{-\frac{1}{2}j (j-1) t}.
\end{equation}
Comparing this with \eqref{BIR} and \eqref{coeff} we obtain
\[
\frac{\displaystyle \prod_{k=i+1}^r \binom{k}{2}}
{\displaystyle \prod_{k=i, k\neq j}^r \left[\binom{k}{2} - \binom{j}{2}\right]}
= \frac{i}{r}  {r\choose i}^2
\frac{\displaystyle (-1)^{i+j} {r-i\choose j-i}^2}
{\displaystyle {2\,j-2\choose j-i}{j+r-1 \choose r-j}},
\]
valid for any integers $i,j,r$ with $1 \le i \le j \le r$.
This equality can be verified directly by simple algebra.
\end{remark}

\section{Convergence rate to Bernstein's theorem: Voronovskaya's theorem}
\label{Vsec}
An important result in the theory of approximation of continuous
functions is Voronovskaya's theorem \cite{VOR32}.
It is the simplest example of {\em saturation}, namely that, for certain
operators, convergence cannot be too fast even for very smooth functions.
See DeVore and Lorentz\cite[Theorem 3.1]{DL93}.
Voronovskaya's theorem gives a rate of convergence to Bernstein's theorem.
From a probabilistic point of view, the theorem is nothing 
else but the convergence of the generator of the discrete Markov
chain to the generator of the Wright-Fisher diffusion.
We shall not use anything from the theory of generators, but we shall
give an independent probabilistic proof below for
$C^2$ functions $f$, including a slightly improved form under the assumption
that $f''$ is Lipschitz. In this case, its Lipschitz constant is
\[
\Lip(f'') := \sup_{x \neq y} |f''(x)-f''(y)|/|x-y|.
\]
Recall that $\LL f$ is defined by \eqref{generator}.
\begin{theorem}[Voronovskaya, 1932]
\label{vorthm}
For any $f \in C^2[0,1]$, 
\begin{equation}
\label{genconv}
\lim_{n \to \infty} \max_{x \in [0,1]}| n (B_n f(x) - f(x)) - \LL f(x) |=0.
\end{equation}
If moreover $f''$ is Lipschitz then, 
for any $n\in\N$,
\begin{equation}\label{sconv}
\max_{x \in [0,1]}| n (B_n f(x) - f(x)) - \LL f(x) |\le
\frac{\Lip(f'')}{16 \cdot 3^{1/4}} \, n^{-1/2}
\end{equation}
\end{theorem}
\begin{proof}
Using Taylor's theorem with the remainder in integral form,
\[
f(G_n(x)) - f(x) = f'(x) (G_n(x)-x) + \int_x^{G_n(x)} (G_n(x)-t)f''(t)dt.
\]
Since $\int_x^{G_n(x)} (G_n(x)-t) dt = \half (G_n(x)-x)^2$, we have,
from \eqref{STEPvar}, $\E\int_x^{G_n(x)} (G_n(x)-t) dt=\half \frac{x(1-x)}{n}$.
Therefore,
\[
n(\E f(G_n(x)) - f(x)) - \half x(1-x) f''(x)
= n\, \E \int_x^{G_n(x)} (G_n(t)-t)\, (f''(t)-f''(x)) dt
=: n\, \E J_n(x).
\]
We estimate $\E J_n(x)$ by splitting the expectation as
\[
\E J_n(x) = \E[J_n(x); |G_n(x)-x|\le \delta]
+ \E[J_n(x); |G_n(x)-x| > \delta],
\]
where $\delta > 0$ is chosen by the uniform continuity of $f''$:
for $\epsilon > 0$ let $\delta$ be such that
$|f''(x)-f''(y)| \le \epsilon$ whenever $|x-y|\le \delta$.
Thus, $|J_n(x)| \le \epsilon |\int_x^{G_n(x)} (G_n(x)-t) dt|$
and so
\begin{multline*}
|\E [J_n(x); |G_n(x)-x|\le \delta]| 
\le \epsilon \E \left[ \left|\int_x^{G_n(x)} (G_n(x)-t) dt\right|;\,
|G_n(x)-x| \le \delta\right]
\le \epsilon \frac{x(1-x)}{n} \le \frac{\epsilon}{4n}.
\end{multline*}
On the other hand, with $\|f''\|=\max_{0\le x \le 1} |f''(x)|$,
we have $|J_n(x)| \le 2 \|f''\| 
|\int_x^{G_n(x)} (G_n(t)-t)dt | \le 2\|f''\|$, so
\[
|\E [J_n(x); |G_n(x)-x|> \delta]| 
\le 2 \|f''\| \P(|G_n(x)-x|> \delta)
\le 4 \|f''\| e^{-\delta n^2/2},
\]
by \eqref{STEPexceed}.
Hence
\[
|n(\E f(G_n(x)) - f(x)) - \half x(1-x) f''(x)|
\le \frac{\epsilon}{4} + 4 \|f''\| n e^{-\delta n^2/2}.
\]
Letting $n \to \infty$, and since $\epsilon >0$ is arbitrary, we
obtain the first claim \eqref{genconv}.
\\
\S Assume next that $f''$ is Lipschitz with Lipschitz constant $\Lip(f'')$.
Then
\begin{align}
|n(\E f(G_n(x)) - f(x)) - \half x(1-x) f''(x)|
& \le n \E  \int_{x\wedge G_n(x)}^{x\vee G_n(x)} |G_n(x)-t|\, 
|f''(t)-f''(x)| dt 
\nonumber
\\
& \le n \Lip(f'')\, 
\E  \int_{x\wedge G_n(x)}^{x\vee G_n(x)} |G_n(x)-t| \,|t-x| dt 
\nonumber
\\
& =\frac{1}{6} n \Lip(f'')\, \E |G_n(x)-x|^3
\nonumber
\\
& \le \frac{1}{6} n \Lip(f'')\, (\E (G_n(x)-x)^4)^{3/4}
\label{mmm}
\end{align}
We have $\E (G_n(x)-x)^4 = n^{-4} \E (S_n - nx)^4$, where
$S_n$ is a binomial random variable--see \eqref{binomdist}.
We can easily find (or look in a textbook) that
\[
\E(S_n-nx)^4 = n x(1-x) (1-6x+6x^2 + 3nx-3nx^2) =: \mu(n,x).
\]
Since $\mu(n,x)=\mu(n,1-x)$ and, for $n \ge 2$, the function
$\mu(n,x)$ is concave in $x$, it follows that $\mu(n,x)
\le \mu(n, 1/2) = n(2n-1)/16 \le 3n^2/16$.
On the other hand, for $n=1$, $\mu(1,x) \le 3/16$ for all $x$.
Thus, the last term of \eqref{mmm} is upper bounded by
\[
\frac{1}{6} n \Lip(f'') \left(\frac{\mu(n,x)}{n^4}\right)^{3/4}
\le \frac{\Lip(f'')}{6} n^{-1/2} (3/16)^{3/4} = \frac{\Lip(f'')}{16 \cdot 3^{1/4}} n^{-1/2}.
\]
\end{proof}

\begin{remark}
Probabilistically, the operator $B_n-I$, where $I$ is the identity
operator, maps a function $f \in C[0,1]$ to the function 
$y \mapsto \E f((G_n(y))-f(y) 
= \E\big[ f(H^{k+1}_n(x)) - f(H^k_n(x)) | H^k_n(x)=y\big]$ 
which is the expected change of the function $f$
(under the action of the chain) per unit time, conditional on the current state
being equal to $y$. Since the natural time scale is counted in time units
that are multiples of $1/n$, we can interpret $n(B_n-I)f$ as the expected change
of $f$ per unit of real time.
Thus, its ``limit'' $\LL$ should play a similar role for the diffusion.
And, indeed, it does, but we shall not use this here.
For further information on diffusions and their generators see,
e.g., Karlin and Taylor \cite{KT81}.
\end{remark}


\section{Joint limits}
\label{JOINTsec}
The goal of this section is a probabilistic
proof of Theorem \ref{JOINTLIMSthm}.
First notice that it suffices to prove \eqref{fidiconv}
for polynomial functions. Indeed, if $f$ is continuous on
$[0,1]$ and $\epsilon >0$,
there is a polynomial $h=B_k f$ such that $\|h-f\|\le \epsilon$
(by Bernstein's Theorem \ref{krthm}).
But then
\[
\|B^{[nt]}_n f - f\| \le 
\|B_n^{[nt]} f - B_n^{[nt]} h\| 
+\|\sP_t f - \sP_t h\| 
+\|B_n^{[nt]} h - \sP_t h\| \le
2\epsilon + \|B_n^{[nt]} h - \sP_t h\|,
\]
where we used the fact that both $B_n^{[nt]}$ and $\sP_t$ are defined
via expectations, and so $\|B_n^{[nt]} f - B_n^{[nt]} h\|
= \|B_n^{[nt]}(f-h)\| \le \|f-h\|$, and, similarly,
$\|\sP_t f - \sP_t h\| = \|\sP_t (f-h)\| \le \|f-h\|$.
Therefore, if $\|B_n^{[nt]} h - \sP_t h\| \to 0$ for
polynomial $h$ then  $\|B^{[nt]}_n f - f\| \to 0$ for any
continuous $f$.
Equivalently, we need to prove that
\begin{equation}
\label{fidiconvpoly}
\lim_{n \to \infty} \E h(H_n^{[nt]}(x)) = \E h(X_t(x)), \text{
uniformly in $x$}.
\end{equation}

Notice that $H_n^{[nt]}(x)$, $t \ge 0$, is not a Markov process.
However if $\Phi(t)$, $t\ge 0$, is a standard Poisson
process with rate $1$, starting from $\Phi(0)=0$, and independent
of everything else, then
\[
X^{(n)}_t(x) := H_n^{\Phi(nt)}(x), \quad t \ge 0,
\]
is a Markov process for each $n$.
Moreover, for all $f \in C^2[0,1]$,
\begin{equation}
\label{Apoi}
\left|\E f(H^{\Phi(nt)}_n(x)) - \E f(H^{[nt]}_n(x))\right| \to 0,
\quad \text{uniformly in $x$}.
\end{equation}
To see this, let $G_n^1, G_n^2, \ldots$ be i.i.d.\ copies of $G_n$,
as in \eqref{Amarkov}, and write the triangle inequality
\begin{equation}
\label{Gtri}
\left|\E f(G_n^2\comp G_n^1(x)) - f(x)\right| \le
\left|\E [f(G_n^2\comp G_n^1(x)) -  f(G_n^1(x))]\right|
+ \left|\E [f(G_n^1(x)) - f(x)]\right|.
\end{equation}
Since $\left|\E [f(G_n^2(y)) -  f(y)]\right| \le \|B_n f - f\|$ for all $y$,
and since $G^1$ is independent of $G^2$, we have that the first term
of the right side of \eqref{Gtri} is $\le \|B_n f - f\|$,
and so 
\[
\left|\E f(G_n^2\comp G_n^1(x)) - f(x)\right| \le 2 \|B_n f -f \|,
\quad \text{for all } 0 \le x \le 1.
\]
By the same argument, for $k< \ell$,
\[
\left|\E f(G_n^\ell \comp \cdots \comp G_n^{k+1}(y)) - f(y)\right|
\le (\ell-k) \|B_n f - f\|,
\quad \text{for all } 0 \le y \le 1.
\]
Since $H_n^k(x)$ is independent of $G_n^\ell \comp \cdots \comp G_n^{k+1}$,
we can replace the $y$ in the last display by $H_n^k(x)$ and obtain
\[
\left|\E [f(H_n^\ell(x)) - f(H_n^k(x))]\right| \le (\ell-k) \|B_n f - f\|,
\quad \text{for all } 0 \le x \le 1.
\]
Using the fact that the Poisson process $\Phi$ is independent
of everything else, we obtain
\begin{equation}
\label{chp}
\left|\E\{f(H_n^{\Phi(nt)}(x))-f(H_n^{[nt]}(x))\} \right|
\le \E\{|\Phi(nt)-[nt]|\}\,\|B_n f - f\|,
\quad \text{for all } 0 \le x \le 1.
\end{equation}
But $\E\{|\Phi(nt)-[nt]|\} 
\le \E\{|\Phi(nt)-nt|\} + 1
\le \sqrt{\E\{|\Phi(nt)-nt|\}^2} + 1
= \sqrt{nt}+1$,
while, from Voronovskaya's theorem, $n \|B_n f -f \| \to \|\LL f\|$.
Therefore the right-hand side of \eqref{chp} converges to $0$ as 
$n \to \infty$, and this proves \eqref{Apoi}.

Therefore \eqref{fidiconvpoly} will follow from
\begin{equation}
\label{fidiconv3}
\lim_{n \to \infty} \E h(X^{(n)}_t(x)) = \E h(X_t(x)),
\quad \text{uniformly in } x,
\end{equation}
for all polynomial $h$.

For each $s \ge 0$, define a random curve $Y^{(s)}_t$ that
follows $X^{(n)}$ up to time $s$ and then switches to $X$.
More precisely, define
\begin{equation}
\label{XX}
Y_t^{(s)} := \begin{cases}
X^{(n)}_t(x), & 0 \le t \le s \\
X_{t-s}(X^{(n)}_s(x)), & t \ge s
\end{cases}.
\end{equation}
See Figure \ref{gluing}.
\begin{figure}[t]
\begin{center}
\epsfig{file=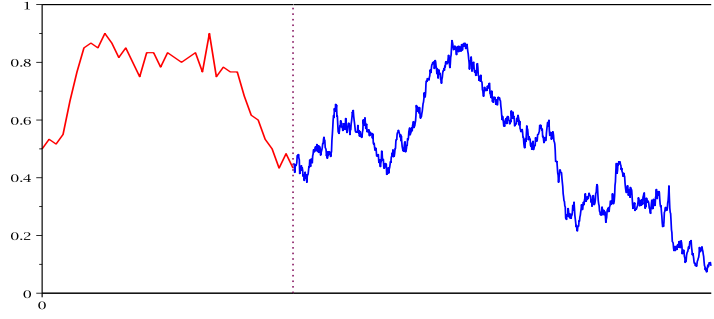,width=9.5cm,height=4.3cm}
\begin{minipage}{12cm}
\caption{\small\em 
A trajectory of (joined-by-straight-lines version of) the Markov chain $X\n$ 
followed by the diffusion $X$.
}
\label{gluing}
\end{minipage}
\end{center}
\end{figure}
It is here assumed that the Brownian motion $W$ 
driving the defining equation \eqref{SDEX} for the Wright-Fisher
diffusion is independent of all random variables
used for the construction of $X^{(n)}$.
Since, for any given initial state, the solution to \eqref{SDEX}
is unique, we may replace the initial state by a random variable
independent of the Brownian motion, and this is what we did in
the last formula.
Thus, if we prove that $\E h( Y^{(s)}_t)$ is differentiable with respect to $s$,
we shall have, for $t \ge s$,
\begin{equation}
\label{HOMO}
\E h(X_t(x)) - \E h(X^{(n)}_t(x)) 
= 
\E h(Y^{(0)}_t) - \E h(Y^{(t)}_t) 
= \int_0^t \pd{}{s} \E h(X_{s}(X^{(n)}_{t-s}(x)))\, ds.
\end{equation}
To show that the last derivative exists as well as estimate it, we
estimate $\pd{}{s} \E h(X_{s}(X^{(n)}_t(x)))$
and $\pd{}{t} \E h(X_{s}(X^{(n)}_t(x)))$.

Since $h$ is a polynomial, it follows that $\sP_s h(x) =
\E h(X_s(x))$ is  also a polynomial function of $x$
(see Proposition \ref{POLY}) and so the forward equation \eqref{FRONTeq}
holds:
\[
\pd{}{s} \E h(X_s(x)) = \LL \sP_s h(x).
\]
Therefore, for any random variable $Y$ with values in $[0,1]$
and
independent of $X_s$, we have
\[
\pd{}{s} \E \{h(X_s(Y))|Y\} = \LL \sP_s h(Y).
\]
Taking expectations of both sides and interchanging differentiation
and expectation (by the DCT) we have
\[
\pd{}{s} \E h(X_s(Y)) = \E \LL \sP_s h(Y).
\]
Setting $Y=X_t\n(x)$, we have
\begin{equation}
\pd{}{s} \E h(X_s(X\n_t(x))) = \E \LL (\mathsf P_s h)(X\n_t(x)).
\label{BACKb}
\end{equation}

Assume next that $f: [0,1]\to \R$ is any function. 
Using the identity
\[
\E [f(H_n^{k+1}(x)) - f(H_n^k(x)) | H_n^k(x)=y] = \E f(G_n(y)) - f(y)
= B_n f(y)-f(y),
\]
valid for any function $f$ and any $x$, together with the
fact that
$\P(\Phi(t+h) - \Phi(t)=1) = h + o(h)$, as $h \downarrow 0$,
we arrive at
\[
\pd{}{t} \E f(X^{(n)}_t(x)) 
= n \E [B_n f(X_t\n(x))-f(X_t\n(x))].
\]
Setting now $f = \sP_s h$ and observing that 
$\E (\sP_s h) (X_t\n(x)) = \E h(X_s(X_t\n(x)))$,
we arrive at
\begin{equation}
\pd{}{t} \E  f(X_s(X^{(n)}_t(x)))
= n \E[ B_n \sP_s h(X_t\n(x))-\sP_s h(X_t\n(x))] 
\label{FRONTb}
\end{equation}

Combining \eqref{BACKb} and \eqref{FRONTb} we have a formula
for the derivative appearing in the last term of \eqref{HOMO}:
\[
\pd{}{s} \E h(X_s(X^{(n)}_{t-s}(x)))
= \E F(X\n_{t-s}(x)),
\]
where
\[
F(y):= \LL(\mathsf P_s h)(y) 
- n[  B_n \mathsf P_s h(y) - (\mathsf P_s h)(y)]
\]
Assume now that $(\mathsf P_s h)''$ is Lipschitz with Lipschitz constant $L_s$.
Then, by \eqref{sconv},
\[
|F(y)| \le c_2 L_s n^{-1/2}, \quad 0 \le y \le 1,
\]
where $c_2=1/16\cdot 3^{1/4}$ and $L_s = \Lip((\sP_s h)'')$, and so
\[
\left|\E h(X\n_t(x)) - \E h(X_t(x)) \right|
\le 
\int_0^t \left|\E F (X\n_{t-s}(x)\right|\, ds 
\le c_2 n^{-1/2} \int_0^t L_s ds.
\]
By the formula for $\sP_s h$ when $h$ is a polynomial
(Proposition \ref{POLY}), it follows that the $\int_0^t L_s ds$
is a finite constant.
Hence \eqref{fidiconv3} has been proved.
\qed

\begin{corollary}
\label{cor1}
With $f_r(x) = x^r$,
\[
\lim_{n \to \infty} B_n^{[nt]} f_r(x) = \sum_{i=1}^r b_{i,r}(t)\, x^i,
\]
where the $b_{i,r}(t)$ are given by \eqref{BIR2}.
\end{corollary}
This is Theorem 2 in Kelisky and Rivlin \cite{KR67}.

\begin{corollary}
With $f_\theta(x) = e^{-\theta x}$, we have
\[
\lim_{n \to \infty} B_n^{[nt]} f_\theta(x)
= \E e^{-\theta X_t(x)} =: H(t,x,\theta),
\]
where $H(t,x,\theta)$ satisfies
$H(0,x,\theta)= e^{-\theta x}$ and the PDE
\[
\frac{\partial H}{\partial t} =
-\frac{\theta^2}{2} \frac{\partial H}{\partial \theta}
-\frac{\theta^2}{2} \frac{\partial^2 H}{\partial \theta^2}.
\]
\end{corollary}
The solution to this PDE can be expressed in terms of modified Bessel
functions. We shall not pursue this further here.

\section{Further comments}
We provided a fully stochastic explanation of the phenomenon
of convergence of $k$ iterates of Bernstein operators of degree $n$
when $n$ and $k$ tend to infinity in different ways. This problem
has received attention in the theory of approximations of continuous
functions. We showed that the problem can be interpreted naturally
via stochastic processes. In fact, these processes, the Wright-Fisher
model and Wright-Fisher diffusion are very basic in probability theory
and are well-understood.

There are a number of interesting directions that open up. 
The most crucial thing is that $B_n f(x)$ is the expectation of
a random variable. We can construct different operators
by using different random variables. See, e.g., Karlin and Ziegler
\cite[Eq.\ (1.5)]{KZ70} for an operator related to a Poisson
random variable.
Whereas Karlin and Ziegler study iterates of these operators, 
their approach is more analytical than probabilistic. 
By using approximations by stochastic differential equations,
and taking advantage of the tools stochastic calculus,
it is possible to derive convergence rates and other interesting results,
including explicit formulas, such as the formula for 
$\lim_{n \to \infty}B_n^{[nt]} f_r(x)$ (Corollary \ref{cor1} and formula 
\eqref{BIR2}), obtained here by a simple application of the It\^o formula.

In Section \ref{interlude} we explained the most standard Wright-Fisher
model where mutations are not allowed. 
If we assume that the probability that a gene of one type
changing to another type also depends on the 
number of genes of each type, in a possibly nonlinear fashion,
we obtain a more general model.
Mathematically, this is captured by letting
$h:[0,1]\to [0,1]$ be an appropriate function,
and by considering the Markov chain obtained by
iterating independent copies of the function $x\mapsto G_n(h(x))$,
that is the Markov chain 
$G_n^k \comp h \comp \cdots \comp G_n^1 \comp h(x)$, $k\in \N$.
For the case $h(x) =ax+b$ with $a,b$ chosen so that $0\le h(x) \le 1$,
see Ethier and Norman \cite{EN77}.



\appendix
\section*{Appendix}

\subsection*{A1$\quad$ Composing random maps}
By a random map $G : T\to S$ we mean a random function
from some probability space $\Omega$ into a subspace of $S^T$ of
functions from $T$ to $S$.
In rigorous probability theory, this means that all sets involved
are equipped with $\sigma$-algebras in a way that $G(\omega) \in T^S$ is 
a measurable function of $\omega$. As a concrete example,
in this paper, we considered $S=T=[0,1]$ and $\Omega = [0,1]^n$.
We equipped all sets with the Lebesgue measure. For $\omega=
(\omega_1, \ldots, \omega_n) \in \Omega$, the map $U_i: \omega \mapsto \omega_i$
is a random variable with the uniform distribution. Moreover,
$U_1, \ldots, U_n$ are independent.
We defined $G(\omega)$ the function
\[
G(\omega)(x) = \frac{1}{n} \sum_{i=1}^n \1_{\omega_i \le x}.
\]
By taking a different $\Omega$, we are able to construct two (or more)
random maps $G_1, G_2$ that are independent. 
As usual in probability, we suppress the symbol $\omega$ from the definition
of $G$. When we talk about the composition $G_2 \comp G_1$
of two random functions as above, we are talking about the composition
with respect to $x$. That is,
$G_2\comp G_1(\omega)(x) := G_2(\omega)(G_1(\omega)(x))$.
If $f$ is a deterministic function we define the operator $f \mapsto T_i f$
by $T_i f(x) := \E f(G_i(x))$, $i=1,2$.
We can then easily see that
\[
\E f(G_2\comp G_1(x)) = T_1 \comp T_2 f(x),
\]
the point being that the order of composition outside the
expectation is the reverse of the one inside.
This is why, for instance, $\E f(G_n^k \comp \cdots \comp G_n^1(x))
= B_n \comp \cdots \comp B_n f(x)$ in eq.\ \eqref{Bkf}.
Another example of a random map is the $x \mapsto X^{(n)}_s(x)$,
where $X^{(n)}$ is the Markov chain constructed in Section \ref{JOINTsec}.
For fixed $s$ and $n$, the initial state $x$ is mapped into the
state $X^{(n)}_s(x)$ at time $s$. 
And yet another random map is $x \mapsto X_t(x)$, where $X_t$, $t \ge 0$, is the
Wright-Fisher diffusion. 
We composed these random maps in \eqref{XX}, after assuming that 
they are independent.

\subsection*{A2$\quad$ Hoeffding's inequality}
Let $X_1, \ldots, X_n$ be independent random variables with zero mean
with values between $-c$ and $c$ for some $c >0$.
Then, for $-c \le x \le c$,
\begin{equation}
\label{hoeffding}
\P(|S_n|> nx) \le 2 e^{-n x^2/2c}.
\end{equation}
This inequality, due to Hoeffding \cite{HOE63}, 
is very well-known and can be found in many probability 
theory books. We just prove it below for completeness.
\begin{proof}
Since $e^{\theta x}$ is a convex function of $x$
for all $\theta > 0$ and so if we write $X_i$ as a convex combination
of $-c$ and $c$,
\[
X_i = \frac{c-X_i}{2c} (-c) + \frac{X_i+c}{2c}  c,
\] 
we obtain
\[
e^{\theta X_i}\le\frac{c-X_i}{2c} e^{-\theta c} + \frac{X_i+c}{2c} e^{\theta c},
\]
Hence
\[
\E e^{\theta X_i} 
= \frac{1}{2} e^{-\theta c} + \frac{1}{2} e^{\theta c}
=\cosh (\theta c) \le e^{\theta^2 c^2/2}.
\]
Here we used the inequality $\cosh t \le e^{t^2/2}$, valid for all real $t$.
This implies that $\E e^{\theta S_n} = \prod_i \E e^{\theta X_i}
\le e^{n \theta^2 c^2/2}$.
Hence, for any $\theta > 0$,
\[
\P(S_n > nx)=\P(e^{\theta S_n} > e^{\theta n x}) 
\le e^{-\theta n x}\E e^{\theta S_n}
\le e^{-\theta n x} e^{ n\theta^2 c^2/2}
= e^{\theta n(\theta c^2/2-x)}.
\]
The last exponent is minimized for $\theta=\theta^* = x/c^2$ and its minimum
value is $\theta^* n(\theta^* c^2/2-x) = -x^2 n/2c^2$. Hence
$\P(S_n > nx) \le e^{-n x^2/2c}$. Reversing the roles of $X_i$ and $-X_i$,
we have $\P(-S_n < -nx) \le e^{-n x^2/2c}$ also.
\end{proof}


\subsection*{A3$\quad$  Convergence of Markov chains to diffusions}
Recall that sequence of real-valued random variables $Z:=(Z_k, k =0,1,\ldots)$
is said to be a time-homogeneous Markov chain if
$\P(Z_{k+1} \le y \mid Z_k=x, Z_{k-1} \ldots, Z_0)
= \P(Z_{k+1} \le y \mid Z_k=x)
= \P(Z_1 \le y\mid Z_0=x)$, for all $k$, $x$ and $y$.
Often, Markov chains depend on a parameter which, without loss of
generality, we can take to be an integer $n$.
Let $Z_n:=(Z_{n,k}, k =0,1,\ldots)$, $n=1,2,\ldots$ be such a sequence.
Frequently, it makes sense to consider this process at a time scale
that depends on $n$. The problem then is to choose (if any) a sequence $\tau_n$
of positive real numbers such that $\lim_{n \to \infty} \tau_n = \infty$
and study instead the random function
\[
Z_{n,{[\tau_n t]}}, \quad t \ge 0,
\]
which, hopefully, may have a limit in a certain sense.
This can be made continuous by linear interpolation. That is, consider
the random function $Z_n(t)$, $t \ge 0$, defined by
\begin{equation}
\label{Xinterp}
Z_n(t) := Z_{n,{[\tau_n t]}}
+ (\tau_n t-[\tau_n t]) (Z_{n,{[\tau_n t]+1}}-Z_{n,{[\tau_n t]}}).
\end{equation}
We seek conditions under which the sequence of random continuous
functions $(Z_n(t), t \ge 0)$ converges weakly to a 
random continuous function $(Z(t), t \ge 0)$.

To define weak convergence, we first define the notion of convergence
in $C[0,\infty)$ by saying that a sequence of continuous functions $f_n$ converges
to a continuous function  (write $f_n \to f$) if $\sup_{0\le t \le T}|f_n(t)-f(t)| \to 0$,
as $n \to \infty$, for all $T \ge 0$.
Then we say that $\phi:C[0,\infty) \to \R$ is continuous if,
for all continuous functions $f$, $\phi(f_n) \to \phi(f)$
whenever $f_n \to f$.
Finally, we say that the sequence of random continuous functions $Z_n$ converges
weakly \cite{EK86} to the random continuous function $Z$ if $\E \phi(Z_n) \to \E \phi(Z)$,
as $n \to \infty$, for all continuous $\phi:C[0,\infty) \to \R$.

%
We now quote, without proof, a useful theorem that enables one to deduce
weak convergence to a random continuous function that satisfies a stochastic
differential equation. For the whole theory we refer to Stroock and Varadhan
\cite[Chapter 11]{SV79}.
\begin{theorem} 
\label{MC2DIFF}
Let, for each $n \in \N$,
$Z_{n,k}$, $k=0,1,\ldots$, be a sequence of real random variables forming
a time-homogeneous Markov chain. 
Assume there is a sequence $\tau_n$, with $\tau_n \to \infty$, such that
\begin{equation}
\label{STEPVAR}
\tau_n\, \E[(Z_{n,k+1}-Z_{n,k})^2 \mid Z_{n,k}\n=x] \to \sigma^2(x),
\end{equation}
uniformly over $|x| \le R$ for all $R>0$,
for some continuous function $\sigma^2(x)$.
For the same $\tau_n$, we also assume that
\begin{equation}
\label{STEPMEAN}
\tau_n\, \E[Z_{n,k+1}-Z_{n,k} \mid Z_{n,k}=x] \to b(x),
\end{equation}
uniformly over $|x| \le R$ for all $R>0$,
for some continuous function $b(x)$. Assume also that,
for all $R> 0$, there are positive constants $c_1, c_2, p$, such that
\begin{equation}
\label{STEPEXCEED}
\tau_n \P( |Z_{n,k+1}-Z_{n,k}| > \epsilon \mid Z_{n,k}=x)
\le c_1 e^{-c_2 \epsilon^p n},
\end{equation}
for all $\epsilon > 0$ and all $|x| \le R$.
Finally, assume that there is $x_0\in \R$ such that
$\P(|Z_{n,0} - x_0|>\epsilon) \to 0$, for all $\epsilon > 0$.
Then, as $n \to \infty$, the sequence of random continuous functions 
defined as in \eqref{Xinterp} converges weakly to
the solution of the stochastic differential equation
\[
dZ(t) = b(Z(t)) dt + \sigma(Z(t)) d W_t, \quad Z_0=x_0,
\]
provided that this equation admits a unique strong solution.  
\end{theorem}
This conditions in Theorem \ref{MC2DIFF} are much stronger
than those of \cite[Theorem 11.2.3]{SV79}.
However, it is often the case that the conditions can be verified.
This is indeed the case in this paper.

\bigskip
\noindent{\em Acknowledgments.}
The authors would like to thank Andrew Heunis and Svante Janson
for their comments on this paper.

\end{document}